\documentclass[11pt, reqno]{amsart}

\usepackage{amsmath, amsthm, amscd, amsfonts, amssymb, graphicx, color}
\usepackage[bookmarksnumbered, colorlinks, plainpages]{hyperref}
\input{mathrsfs.sty}

 \makeatletter
\let\reftagform@=\tagform@
\def\tagform@#1{\maketag@@@{(\ignorespaces\textcolor{blue}{#1}\unskip\@@italiccorr)}}
\renewcommand{\eqref}[1]{\textup{\reftagform@{\ref{#1}}}}
\makeatother
\usepackage{hyperref}
\hypersetup{colorlinks=true, linkcolor=red, anchorcolor=green,
citecolor=cyan, urlcolor=red, filecolor=magenta, pdftoolbar=true}

\newtheorem{theorem}{Theorem}[section]
\newtheorem{lemma}[theorem]{Lemma}
\newtheorem{proposition}[theorem]{Proposition}
\newtheorem{corollary}[theorem]{Corollary}
\theoremstyle{definition}

\newtheorem{example}[theorem]{Example}

\theoremstyle{remark}

\numberwithin{equation}{section}

{\normalsize }
\def\u|{|\kern-0.1em|\kern-0.1em|}
\def\U|{\Big|\kern-0.1em\Big|\kern-0.1em\Big|}

\def\~{\hskip-2pt}

\def\<{\langle}
\def\>{\rangle}

\def\VV{\lower-0.1ex\hbox{$\ \begin{matrix}\vee\\[-2ex]\vee\end{matrix}\ $}}
\def\vv{\lower-0.2ex\hbox{$\ \begin{matrix}\wedge\\[-2ex]\wedge\end{matrix}\ $}}

\def\({\left(}
\def\){\right)}

\def\u|{|\kern-0.1em|\kern-0.1em|}
\def\U|{\Big|\kern-0.1em\Big|\kern-0.1em\Big|}

\numberwithin{equation}{section}

\def\tr{\mathrm{Tr}\;}

\def\phi{\varphi}
  \def\etal{et al.\,}

\begin{document}


\title[Superquadratic trace  functions]{Klein's trace inequality and superquadratic  trace  functions}

\author[M. Kian \MakeLowercase{and} M.W. Alomari]{Mohsen Kian  \MakeLowercase{and}   Mohammad W. Alomari}

\address{Mohsen Kian: Department of Mathematics, University of Bojnord, P.O. Box
1339, Bojnord 94531, Iran}
\email{kian@ub.ac.ir }

\address{Mohammad W. Alomari: Department of Mathematics, Faculty of Science and Information	Technology, Irbid National University,  P.O. Box 2600, Irbid, P.C. 21110, Jordan.}
	\email{mwomath@gmail.com}

\subjclass[2010]{Primary: 47A56,  15A45  Secondary:  15A18, 15A42.}

\keywords{Klein's trace inequality, superquadratic trace function, majorization}

\begin{abstract}
 We show that if $f$ is a non-negative superquadratic function, then $A\mapsto\mathrm{Tr}f(A)$ is a superquadratic function on the matrix algebra. In particular,
   \begin{align*}
	\tr f\left( {\frac{{A + B}}{2}} \right) +\tr f\left(\left| {\frac{{A - B}}{2}}\right|\right) \leq  	\frac{{\tr {f\left( A \right)}   + \tr  {f\left( B \right)}  }}{2}
\end{align*}
  holds for all positive matrices $A,B$.
  In addition, we present a Klein's inequality for superquadratic functions as
$$
	\mathrm{Tr}[f(A)-f(B)-(A-B)f'(B)]\geq \mathrm{Tr}[f(|A-B|)]
$$
for all positive matrices $A,B$.
It gives in particular an improvement  of the Klein's inequality for non-negative convex function.
  As a consequence, some  variants of the Jensen trace inequality for  superquadratic functions have been presented.
\end{abstract}

\maketitle

\section{Introduction and Preliminaries}

In study of quantum mechanical systems, there are many famous  concepts which are related to the trace function  $A\mapsto\mathrm{Tr}(A)$. The well-known relative entropy of  a density matrix $\rho$ (a positive matrix of trace one) with respect of    another  density matrix $\sigma$ is defined by
$$S(\rho|\sigma)=\mathrm{Tr}(\rho\log\rho)-\mathrm{Tr}(\rho\log\sigma).$$
More generally, for a proper (continuous) real function $f$,    the study of the mapping  $A\mapsto\mathrm{Tr}(f(A))$ is important.

The main subject of this paper, is to study this mapping for a class of real functions, the superquadrtic functions. It is known that  if $f: \mathbb{R} \to \mathbb{R}$  is a continuous    convex (monotone  increasing) function, then  the trace function  $A\mapsto \tr\left(f\left(A\right)\right)$ is a convex (monotone  increasing) function, see  \cite{HJ,EHL}. In Section 2, we present this result for superquadratic functions.

For all Hermitian $n\times n$ matrices $A$ and $B$ and all differentiable convex functions $f: \mathbb{R} \to \mathbb{R}$ with derivative $f^{\prime}$, the well known   Klein  inequality reads as  \begin{align}
\label{eq4.2}\tr\left[  {	{f\left( A \right) - f\left( B \right) - \left( {A - B} \right)f'\left( B \right)}} \right]\ge0.
\end{align}
With $f(t) = t \log t$ $(t>0)$, this gives
\begin{align*}
S\left( {A|B} \right) = \tr A\left( {\log A - \log B} \right) \ge \tr\left( {A - B} \right)
\end{align*}
for positive   matrices $A,B$.    If $A$ and $B$ are density matrices,  then $S\left(A,B\right)\ge0$. This  is  a  classical  application  of  the  Klein  inequality. See \cite{Ca,PZ}. To see a collection of  trace inequalities the reader can refer to \cite{CaLi, FL,FKY,  Hi,ShAb,Ya} and references therein.

In Section 3,  we present a Klein  trace inequality for superquadrtic functions. We show that our result improves previous results in the case of non-negative functions. In-addition, some applications of our results present  counterpart to some known trace inequalities. We give some examples to clarify our results.

\bigskip

Let $\mathscr{B}\left( \mathscr{H}\right) $ be the $C^*$-algebra
of all bounded linear operators defined on a complex Hilbert space
$\left( \mathscr{H};\left\langle \cdot ,\cdot \right\rangle
\right)$  with the identity operator  $I$.   When $\dim \mathscr{H}=n$, we identify $\mathscr{B}\left( \mathscr{H}\right)$
with the algebra $\mathbb{M}_{n}$ of $n$-by-$n$ complex
matrices. We denote by $\mathbb{H}_n$ the real subspace of Hermitian matrices and by $\mathbb{M}^{+}_{n}$   the cone of   positive (semidefinite) matrices. The identity matrix of any size will be denoted by $I$.

Every Hermitian matrix  $A\in\mathbb{H}_n$ enjoys the spectral  decomposition $A=\sum_{j=1}^{n}\lambda_j P_j$, where $\lambda_j$'s are eigenvalues of $A$ and $P_j$'s are projection matrices with $\sum_{j=1}^{n} P_j=I$. If $f$ is a continuous real function which is defined on the set of eigenvalues of $A$, then $f(A)$ is the matrix defined using  the spectral  decomposition  by $f(A)=\sum_{j=1}^{n}f(\lambda_j) P_j$. The eigenvalues of $f(A)$ are just $f(\lambda_j)$. Moreover, If $U$ is a unitary matrix, then $f(U^*AU)=U^*f(A)U$.

For  $A=[a_{ij}]\in \mathbb{M}_{n}$ the canonical trace of $A$ is denoted by $\mathrm{Tr} A$ and is defined to be $\sum_{j=1}^{n}a_{ii}$.   The canonical trace  is a unitary invariant mapping, say $\mathrm{Tr} UAU^*=\mathrm{Tr} A$ for every unitary matrix $U$. So, when $\lambda_1,\cdots,\lambda_n$ are eigenvalues of $A$ and $\left\{\mathrm{\mathbf{u}}_1,\cdots,\mathrm{\mathbf{u}}_n\right\}$  is an orthonormal set of corresponding eigenvectors in $\mathbb{C}^n$, then
$$\mathrm{Tr} A=\sum_{j=1}^n{\lambda_j}(A)=\sum_{j=1}^n{\left\langle A\mathrm{\mathbf{u}}_j,\mathrm{\mathbf{u}}_j\right\rangle}\quad\mbox{and}\quad \mathrm{Tr} f(A)=\sum_{j=1}^nf({\lambda_j}(A))=\sum_{j=1}^nf({\left\langle A\mathrm{\mathbf{u}}_j,\mathrm{\mathbf{u}}_j\right\rangle}).$$

If $\mathscr{H}$ is a separable Hilbert space with  an orthonormal basis $\{e_i\}_i$, an operator  $A\in\mathscr{B}\left( \mathscr{H}\right)$   is said to be  a trace class operator  if
\begin{align*}
\left\| A \right\|_1   = \sum\limits_i {\left\langle {\left( {A^* A} \right)^{1/2} e_i ,e_i } \right\rangle },
\end{align*}
is finite. In this case, the trace of $A$ is defined by $\tr\left( A \right) = \sum\limits_i {\left\langle {Ae_i ,e_i } \right\rangle }$
and  is independent of the choice of the orthonormal basis. When $\mathscr{H}$ is finite-dimensional, every operator is trace class and this definition of trace of $A$ coincides with the definition of the trace of a matrix. \\

For a vector  $\mathrm{\mathbf{x}}=(x_1,\ldots,x_n)$ in $\mathbb{R}^n$, let $\mathrm{\mathbf{x}}^\downarrow$ and $\mathrm{\mathbf{x}}^\uparrow$ denotes  the vectors  obtained  by rearranging entries of $\mathrm{\mathbf{x}}$ in decreasing and increasing order, respectively, i.e., $x_1^\downarrow\geq\ldots\geq x_n^\downarrow$ and $x_1^\uparrow\leq\ldots\leq x_n^\uparrow$.
 A vector $\mathrm{\mathbf{x}}\in\mathbb{R}^n$ is said to be weakly majorised by  $\mathrm{\mathbf{y}}\in\mathbb{R}^n$  and denoted by  $\mathrm{\mathbf{x}}\prec_w \mathrm{\mathbf{y}}$ if $\sum_{j=1}^{k}x_j^\downarrow\leq \sum_{j=1}^{k}y_j^\downarrow$ holds for every $k=1,\ldots,n$. If in addition $\sum_{j=1}^{n}x_j^\downarrow= \sum_{j=1}^{n}y_j^\downarrow$, then $\mathrm{\mathbf{x}}$ is said to be   majorised by  $\mathrm{\mathbf{y}}$ and is denoted by $\mathrm{\mathbf{x}}\prec \mathrm{\mathbf{y}}$. The trace of a vector $\mathrm{\mathbf{x}}\in\mathbb{R}^n$ is defined to be the sum of its entries and is denoted using a same notation as a matrix by $\mathrm{Tr}\  \mathrm{\mathbf{x}}$.

 A matrix $P=[p_{ij}]\in\mathbb{M}_n$ is said to be doubly stochastic if  all of its entries are non-negative and
$$\sum_{i=1}^{n}p_{ij}=1\quad\mbox{for all $j$}\qquad\mbox{and}\qquad \sum_{j=1}^{n}p_{ij}=1\quad\mbox{for all $i$}.$$
For all $\mathrm{\mathbf{x}},\mathrm{\mathbf{y}}\in\mathbb{R}^n$ it is well-known that    $\mathrm{\mathbf{x}}\prec\mathrm{\mathbf{y}}$ if and only if there exists a doubly stochastic matrix $P$ such that $\mathrm{\mathbf{x}}=P\mathrm{\mathbf{y}}$, see \cite[Theorem II.1.10]{bh}. More results concerning majorization can be found in \cite{bh,HJ}.

\bigskip

A function $f:J\subseteq{\mathbb{R}}\to \mathbb{R}$ is called convex if
\begin{align}
f\left(   \alpha t +\left(1-\alpha\right)s \right)\le \alpha f\left(
{t} \right)+ \left(1-\alpha\right) f\left( {s}
\right),\label{eq1.1}
\end{align}
for all points $ s,t \in J$ and all $\alpha\in [0,1]$. If $-f$
is convex then we say that $f$ is concave. Moreover, if $f$ is
both convex and concave, then $f$ is said to be affine.

Geometrically, for all  $x,y\in J$ with $x \le t \le y$, the two points $\left(x,f\left(x\right)\right)$ and $\left(y,f\left(y\right)\right)$  on the graph of $f$ are on or
below the chord joining the endpoints. In symbols, we write
\begin{align*}
f\left(t\right)\le   \frac{f\left( y \right)  - f\left( x \right)
}{y-x}   \left( {t-x} \right)+ f\left( x \right)
\end{align*}
for any $x \le t \le y$ and $x,y\in J$.

Equivalently, given a function $f : J\to \mathbb{R}$, we say that
$f$ admits a support line at $s \in J $ if there exists a $\lambda
\in \mathbb{R}$ such that
\begin{align}
f\left( t \right) \ge f\left( s \right) + \lambda \left( {t - s}
\right) \label{eq1.2}
\end{align}
for all $t\in J$.
The set of all such $\lambda$ is called the subdifferential of $f$
at $s$ and it is denoted by $\partial f$. Indeed, the
subdifferential gives us the slopes of the supporting lines for
the graph of $f$   so that if $f$ is convex, then $\partial f(s) \ne
\emptyset$ at all interior points of its domain.

From this point of view,   Abramovich \etal \cite{SJS} extended the
above idea for what they called superquadratic functions. Namely,
a function $f:[0,\infty)\to \mathbb{R}$ is called superquadratic
provided that for all $s\ge0$ there exists a constant $C_s\in
\mathbb{R}$ such that
\begin{align}\label{eq1.3}
f\left( t \right) \ge f\left( s \right) + C_s \left( {t - s}
\right) + f\left( {\left| {t - s} \right|} \right)
\end{align}
for all $t\ge0$.  A function $f$ is called subquadratic if $-f$ is
superquadratic. Thus, for a superquadratic function we require
that $f$  is  above its tangent line plus a translation of $f$
itself. If $f$ is differentiable and satisfies $f(0) = f^{\prime}(0) = 0$, then one can easily see
that  the constant  $C_s$   in the definition is necessarily $f^{\prime}(s)$,  see \cite{AJS}.

Prima facie, superquadraticity     looks  to be stronger than
convexity,  but if $f$ takes negative values then it
may be considered   weaker. On the other hand,   non-negative subquadratic functions does not  need to be concave. In other words,  there exist    subquadratic function which are convex. This fact helps us first to improve some results for convex functions and second to present some counterpart
 results concerning convex functions.
Some known examples of superquadratic functions are power functions. For every $p\geq2$, the function $f(t)=t^p$ is superquadratic as well as convex. If $1\leq p\leq 2$, then $f(t)=-t^p$  is superquadratic  and concave. To see more examples of superquadratic and subquadratic functions and their properties, the reader can refer to \cite{AJS,SJS,SIP,MA,BPV}. Among others,  Abramovich \etal \cite{SJS} proved that the inequality
\begin{align}\label{eq.SJS}
f\left( {\int {\varphi d\mu } } \right) \le   \int {f\left(
	{\varphi \left( s \right)} \right)-f\left( {\left| {\varphi \left(
			s \right) - \int {\varphi d\mu } } \right|} \right)d\mu \left( s
	\right)}
\end{align}
holds for all probability measures $\mu$ and all nonnegative, $\mu$-integrable functions $\varphi$ if and only if $f$ is superquadratic.

 As a matrix  extension of \eqref{eq.SJS}, Kian \cite{K} showed that if $f:[0,\infty)\to\mathbb{R}$ is a continuous  superquadratic function, then
 \begin{align}\label{qk}
   f(\langle A\mathrm{\mathbf{u}},\mathrm{\mathbf{u}}\rangle)\leq \langle f(A)\mathrm{\mathbf{u}},\mathrm{\mathbf{u}}\rangle-\langle f(|A-\langle A\mathrm{\mathbf{u}},\mathrm{\mathbf{u}}\rangle|)\mathrm{\mathbf{u}},\mathrm{\mathbf{u}}\rangle
 \end{align}
 holds for every positive matrix $A\in\mathbb{M}_n^+$ and every unit vector $\mathrm{\mathbf{u}}\in\mathbb{C}^n$. More generally, it has been shown in \cite{KS} that if $\Phi:\mathbb{M}_n\to\mathbb{M}_m$ is a unital positive linear map, then
  \begin{align}\label{qkd}
   f(\langle \Phi(A)\mathrm{\mathbf{u}},\mathrm{\mathbf{u}}\rangle)\leq \langle \Phi(f(A))\mathrm{\mathbf{u}},\mathrm{\mathbf{u}}\rangle-\left\langle \Phi(f(|A-\langle \Phi(A)\mathrm{\mathbf{u}},\mathrm{\mathbf{u}}\rangle|))\mathrm{\mathbf{u}},
   \mathrm{\mathbf{u}}\right\rangle
 \end{align}
 holds for every positive matrix $A\in\mathbb{M}_n^+$ and every unit vector $\mathrm{\mathbf{u}}\in\mathbb{C}^n$.

\section{Superquadratic trace functions}

It is known that if $f: \mathbb{R} \to \mathbb{R}$ is a continuous convex function, then   the trace function  $A\mapsto \tr\left[f\left(A\right)\right]$ is a convex   function on $\mathbb{M}_n$.   In this section, we present this fact for superquadratic functions. We need some lemmas.
Note that if $\mathrm{\mathbf{x}}=(x_1,\cdots,x_n)\in\mathbb{R}^n$ is a vector and $f:\mathbb{R}\to\mathbb{R}$ is a real function, we denote the vector $\left(f\left(x_1\right),\cdots,f\left(x_n\right)\right)$ by $f(\mathrm{\mathbf{x}})$.

\begin{lemma}\label{lm1}\cite{bh}
 For $\mathrm{\mathbf{x}},\mathrm{\mathbf{y}}\in\mathbb{R}^n$ \\
\rm{ (i)}\ \ If $\mathrm{\mathbf{x}}\prec \mathrm{\mathbf{y}}$, then $|\mathrm{\mathbf{x}}|\prec |\mathrm{\mathbf{y}}|$, where $|\mathrm{\mathbf{x}}|=\left(\left|x_1\right|,\cdots,\left|x_n\right|\right)$.\\
 \rm{(ii)}\ \ $\mathrm{\mathbf{x}}\prec \mathrm{\mathbf{y}}$ if and only if   $\mathrm{Tr} f(\mathrm{\mathbf{x}})\leq \mathrm{Tr} f(\mathrm{\mathbf{y}})$ for every convex function $f$.
\end{lemma}
\begin{lemma}\label{lm2}
  Assume that  $\mathrm{\mathbf{x}},\mathrm{\mathbf{y}}\in\mathbb{R}_+^n$ and $f:\left[0,\infty\right)\to \mathbb{R}$ is a superquadratic   function. If $\mathrm{\mathbf{x}}\prec \mathrm{\mathbf{y}}$, then there exists a doubly stochastic matrix $P$ such that $\mathrm{Tr} f(\mathrm{\mathbf{x}})\leq \mathrm{Tr} f(\mathrm{\mathbf{y}})-\mathrm{Tr} PF$, where $F=\left[f\left(\left|x_i-y_j\right|\right)\right]$.
\end{lemma}
\begin{proof}
  For $\mathrm{\mathbf{x}},\mathrm{\mathbf{y}}\in\mathbb{R}_+^n$, if $\mathrm{\mathbf{x}}\prec \mathrm{\mathbf{y}}$, then there exists a doubly  stochastic matrix $P=[p_{ij}]$ such that $\mathrm{\mathbf{x}}=P\mathrm{\mathbf{y}}$. Therefore, $x_i=\sum_{j=1}^{n}p_{ij}y_j$ for every $i=1,\cdots,n$ and $\sum_{j=1}^{n}p_{ij}=1$. If $f$ is a superquadratic function, then from \eqref{eq.SJS} we conclude that the inequality
  \begin{align}
    f(x_i)=f\left(\sum_{j=1}^{n}p_{ij}y_j\right)\leq\sum_{j=1}^{n}p_{ij}f(y_j)-\sum_{j=1}^{n}p_{ij}
    f\left(\left|y_j-\sum_{j=1}^{n}p_{ij}y_j\right|\right)
  \end{align}
  holds for every $i=1,\cdots,n$. Summing over $i$, we obtain
  $$\mathrm{Tr} f(\mathrm{\mathbf{x}})\leq \mathrm{Tr} f(\mathrm{\mathbf{y}})-\sum_{i,j=1}^{n}p_{ij}     f\left(\left|y_j-x_i\right|\right).$$
  If we put $F=\left[f\left(\left|x_i-y_j\right|\right)\right]$, then $\sum_{i,j=1}^{n}p_{ij}     f\left(\left|y_j-x_i\right|\right)=\mathrm{Tr} PF$. This completes the proof.
\end{proof}

\begin{lemma}\cite{SJS}\label{lm44}
	\label{lemma1}Let $f:[0,\infty)\to\mathbb{R}$ be a superquadratic function. Then
	\begin{enumerate}
		\item $f\left(0\right)\le 0$;
		
		\item If $f$ is differentiable and $f(0)=f^{\prime}(0)=0$, then $C_s=f^{\prime}(s)$ in \eqref{eq1.3} for all $s\ge0$;
		
		\item If $f$ is non-negative, then $f(0)=f^{\prime}(0)=0$ and $f$ is convex and increasing.	\end{enumerate}
\end{lemma}

\begin{theorem}	\label{lemma3}
Let $f:\left[0,\infty\right)\to \mathbb{R}$ be a continuous superquadratic function. If $f$ is non-negative, then the mapping $A\mapsto \mathrm{Tr} \left[f(A)\right]$ is a superquadratic function on $\mathbb{M}_n^+$. More generally,    the inequality
  \begin{align}\label{eq2.1}
	\tr f\left( {\frac{{A + B}}{2}} \right) +\tr f\left(\left| {\frac{{A - B}}{2}}\right|\right) \leq  	\frac{{\tr {f\left( A \right)}   + \tr  {f\left( B \right)}  }}{2}-\mathrm{Tr}[PG+QF]
\end{align}
holds for some     doubly  stochastic matrices $P=[p_{ij}]$ and $Q=[q_{ij}]$, in which $$G=\left[f\left(\frac{1}{2}\left|\left|\lambda_i\right|-\left|\mu_j-\nu_j\right|\right|\right)\right]
\quad\mbox{and}\quad F=\left[f\left(\frac{1}{2}\left|\xi_i-\mu_j-\nu_j\right|\right)\right],$$
  where  $\lambda_i,\xi_i,\mu_i$  and $\nu_i$ are eigenvalues of $A-B$, $A+B$, $A$ and $B$,  respectively.
\end{theorem}
\begin{proof}
  For a Hermitian matrix  $X$, assume that  $\lambda^\downarrow(X)$ and $\lambda^\uparrow(X)$ are eigenvalues of $X$ arranged in decreasing order and increasing order, respectively.  Recall that \cite{bh}  if $A,B$ are Hermitian matrices, then
\begin{align}\label{j}
 \lambda^\downarrow(A)- \lambda^\downarrow(B)\prec\lambda^\downarrow(A-B)\prec \lambda^\downarrow(A)- \lambda^\uparrow(B)
\end{align}
and
\begin{align}\label{jj}
 \lambda^\downarrow(A)+ \lambda^\uparrow(B)\prec\lambda^\downarrow(A+B)\prec \lambda^\downarrow(A)+ \lambda^\downarrow(B).
\end{align}
From \eqref{j} we have
\begin{align}\label{jg}
\lambda^\uparrow(B-A)\prec   \lambda^\downarrow(B)-\lambda^\downarrow(A)
\end{align}
and  noting  Lemma  \ref{lm1} this gives
\begin{align}\label{jgg}
\left|\lambda^\downarrow(A-B)\right|\prec \left|\lambda^\downarrow(B)-\lambda^\downarrow(A)\right|.
\end{align}
 We assume that $\mu_j$ and $\nu_j$ \ $(j=1,\cdots,n)$ are eigenvalues of $A$ and $B$  respectively, arranged  in decreasing order. If  $f$ is superquadratic, then it follows from \eqref{jgg} and Lemma \ref{lm2} that
 \begin{align*}
\mathrm{Tr} f(|A-B|)&=\sum_{j=1}^{n}f\left(|\lambda_j(A-B)|\right)= \mathrm{Tr}\ f\left(\left|\lambda^\downarrow(A-B)\right|\right)\nonumber\\
&\leq \mathrm{Tr}\ f\left(\left|\lambda^\downarrow(B)-\lambda^\downarrow(A)\right|\right)-\mathrm{Tr} PG\qquad(\mbox{by Lemma \ref{lm2}})\nonumber\\
&=\sum_{j=1}^{n} f\left(\left|\mu_j-\nu_j\right|\right)-\mathrm{Tr} PG,
\end{align*}
for some doubly  stochastic matrix $P=[p_{ij}]$, in which $G=\left[f\left(\left|\left|\lambda_i\right|-\left|\mu_j-\nu_j\right|\right|\right)\right]$ and $\lambda_i$'s are eigenvalues of $A-B$. This implies that for every $\alpha\geq0$, the inequality
\begin{align}\label{j2}
  \mathrm{Tr} f(\alpha|A-B|)\leq \sum_{j=1}^{n} f\left(\alpha\left|\mu_j-\nu_j\right|\right)-\mathrm{Tr} P_\alpha G_\alpha
\end{align}
holds   for some doubly  stochastic matrix $P_\alpha$, in which $G_\alpha=\left[f\left(\alpha\left|\left|\lambda_i\right|-\left|\mu_j-\nu_j\right|\right|\right)\right]$ and $\lambda_i$'s are eigenvalues of $(A-B)$. Now suppose that $\alpha\in[0,1]$. Another use of Lemma \ref{lm2} together with  \eqref{jj} gives
\begin{align}\label{jjg}
\mathrm{Tr}\ f\left(\lambda^\downarrow(\alpha A+(1-\alpha)B)\right)  \leq   \mathrm{Tr}\ f\left(\alpha\lambda^\downarrow(A)+ (1-\alpha)\lambda^\downarrow(B)\right)-\mathrm{Tr}QF\
\end{align}
for some doubly  stochastic matrix $Q$, where $F=\left[f\left(\left|\xi_i-\alpha\mu_j-(1-\alpha)\nu_j\right|\right)\right]$ and $\xi_i$'s are eigenvalues of $\alpha A+(1-\alpha)B$.  Therefore
 {\small\begin{align*}
&\mathrm{Tr} f(\alpha A+(1-\alpha)B) \\
&=\sum_{j=1}^{n}f\left(\lambda_j^\downarrow(\alpha A+(1-\alpha)B)\right)\\
&\leq\sum_{j=1}^{n}f\left(\alpha\mu_j+ (1-\alpha)\nu_j\right)-\mathrm{Tr}QF\qquad\qquad(\mbox{by \eqref{jjg}})\\
&\leq \sum_{j=1}^{n}\left\{\alpha f(\mu_j)+(1-\alpha)f(\nu_j)-\alpha f\left((1-\alpha)\left|\mu_j-\nu_j\right|\right) -(1-\alpha)f\left(\alpha\left|\mu_j-\nu_j\right|\right)\right\}-\mathrm{Tr}QF\\
&\hspace{6cm}(\mbox{since $f$ is superquadratic})\\
&=\alpha \mathrm{Tr} f(A)+(1-\alpha)\mathrm{Tr} f(B) -\alpha\sum_{j=1}^{n}f\left((1-\alpha)\left|\mu_j-\nu_j\right|\right)-(1-\alpha)\sum_{j=1}^{n}
f\left(\alpha\left|\mu_j-\nu_j\right|\right)-\mathrm{Tr}QF\\
&\leq\alpha \mathrm{Tr} f(A)+(1-\alpha)\mathrm{Tr} f(B)\\
&\qquad-\alpha\mathrm{Tr} f\left((1-\alpha)\left|A-B\right|\right)-(1-\alpha)\mathrm{Tr} f\left(\alpha\left|A-B\right|\right)-\mathrm{Tr}[(1-\alpha)P_\alpha G_\alpha+\alpha P_{1-\alpha} G_{1-\alpha}+QF],
\end{align*}}
where the last inequality follows from \eqref{j2}. In particular, with $\alpha=1/2$ this gives
  \begin{align*}
	\tr f\left( {\frac{{A + B}}{2}} \right) +\tr f\left(\left| {\frac{{A - B}}{2}}\right|\right) \leq  	\frac{{\tr {f\left( A \right)}   + \tr  {f\left( B \right)}  }}{2}-\mathrm{Tr}[PG+QF]
\end{align*}
for some doubly  stochastic matrices $P=[p_{ij}]$ and $Q=[q_{ij}]$, in which
{\small\begin{align*}
G=\left[f\left(\frac{1}{2}\left|\left|\lambda_i\right|-\left|\mu_j-\nu_j\right|\right|\right)\right]
\quad\mbox{and}\quad F=\left[f\left(\frac{1}{2}\left|\xi_i-\mu_j-\nu_j\right|\right)\right],
\end{align*}}
 where  $\lambda_i$ and $\xi_i$ are eigenvalues of $A-B$ and $A+B$, respectively.  Equivalently
{\small\begin{align*}
	&\tr f\left( {\frac{{A + B}}{2}} \right) +\tr f\left(\left| {\frac{{A - B}}{2}}\right|\right) \\ &\qquad \leq  	\frac{{\tr {f\left( A \right)}  + \tr  {f\left( B \right)}  }}{2} -\sum_{i,j=1}^{n}\left(p_{ij}
f\left(\frac{1}{2}\left|\left|\lambda_i\right|-\left|\mu_j-\nu_j\right|\right|\right)+
q_{ij}f\left(\frac{1}{2}\left|\xi_i-\mu_j-\nu_j\right|\right)\right),
\end{align*}}
from  which we conclude that if $f$ is non-negative, then $A\mapsto \tr f\left(A\right)$ is a superquadratic function.  This completes the proof.
\end{proof}

In 2003,  Hansen \& Pedersen \cite{HP}  proved a  trace version of then Jensen inequality. They showed that if $f:J\subseteq\mathbb{R}\to\mathbb{R}$ is a continuous convex function, then
\begin{align}\label{HPT}
\tr\left[{f\left( {\sum\limits_{i = 1}^k {C_i^* A_i C_i } }
	\right)}\right] \le \tr\left[{\sum\limits_{i = 1}^k {C_i^*
		f\left( {A_i } \right)C_i}}\right]
\end{align}
for every $k$-tuple of Hermitian matrices $(A_1,\cdots, A_k)$   in $\mathbb{M}_n$ with spectra contained in $J$ and every $k$-tuple $\left(C_1, \cdots, C_k\right)$ of  matrices with $\sum_{i=1}^k{C_i^*C_i}=I$.

In the rest of this section, using the concept of superquadratic functions and Theorem \ref{lemma3}, we present   variants of \eqref{HPT}  for superquadratic functions,  which give in particular some  refinements of the Hansen--Pedersen trace inequality \eqref{HPT} in the case of non-negative functions.   Beside our  results concerning \eqref{HPT}, we give a conjuncture as follows.

\textbf{Conjuncture.}\
If $f:[0,\infty)\to\mathbb{R}$ is a continuous superquadratic function, then
{\small\begin{align}\label{HPT-s}
\tr\left[{f\left( {\sum\limits_{i = 1}^k {C_i^* A_i C_i } }
	\right)}\right] \le \tr\left[{\sum\limits_{i = 1}^k {C_i^*
		f\left( {A_i } \right)C_i}}\right]-\tr\left[\sum\limits_{i = 1}^kC_i^*f\left(\left|A_i-\tr \left[\sum\limits_{i = 1}^kC_i^*A_iC_i\right]\right|\right)C_i\right]
\end{align}}
for every $k$-tuple of positive matrices $(A_1,\cdots, A_k)$  in $\mathbb{M}_n^+$  and every $k$-tuple $\left(C_1, \cdots, C_k\right)$ of  matrices with $\sum_{i=1}^k{C_i^*C_i}=I$.


We now use Theorem \ref{lemma3} to present the first  variant of \eqref{HPT} for superquadratic functions.
\begin{corollary}\label{pr3}
Assume that $f:[0,\infty)\to[0,\infty)$ is a continuous  function.  If $f$ is superquadratic, then
\begin{align}\label{JTS}
 \mathrm{Tr}f\left(C^*AC\right)+  \mathrm{Tr} f\left(D^*AD\right)\leq \mathrm{Tr}\left[C^*f(A)C+D^*f(A)D \right]-
  \mathrm{Tr}\left[f\left( \left|DAC\right|\right)+ f\left(\left|C^*AD^*\right|\right)\right]
\end{align}
for every positive matrix $A\in\mathbb{M}_n^+$ and every isometry $C$, where $D=\sqrt{1-CC^*}$.
\end{corollary}
\begin{proof}
 To prove \eqref{JTS}, we apply Theorem \ref{lemma3} and then employ a similar argument as in \cite[Theorem 1.9]{FMPS}. Assume that $A,B\in\mathbb{M}_n^+$.  If $C\in\mathbb{M}_n$ and $C^*C=I$, then   the block matrices $U=\left[\begin{array}{cc}
   C & D\\ 0 & -C^*
 \end{array}\right]$ and  $V=\left[\begin{array}{cc}
   C & -D\\ 0 &  C^*
 \end{array}\right]$ are  unitary matrices in $\mathbb{M}_{2n}$, provided that $D=(I-CC^*)^{1/2}$.
 With $\tilde{A}=\left[\begin{array}{cc}
   A & 0\\ 0&B
 \end{array}\right]$ we compute
 \begin{align}\label{qn1}
  \frac{U^*\tilde{A}U +V^*\tilde{A}V}{2}=\left(C^*AC\right)\oplus \left(DAD+CBC^*\right)
 \end{align}
 and
  \begin{align}\label{qn2}
  \left|\frac{U^*\tilde{A}U -V^*\tilde{A}V}{2}\right|=\left|DAC\right|\oplus \left|C^*AD\right|.
 \end{align}
Now  we use  Theorem \ref{lemma3} to write
 \begin{align*}
  &\mathrm{Tr}f\left(C^*AC\right)+  \mathrm{Tr} f\left(DAD+CBC^*\right)\\
  &=  \mathrm{Tr} f\left(\frac{U^*\tilde{A}U +V^*\tilde{A}V}{2}\right)\qquad\qquad\qquad(\mbox{by \eqref{qn1}})\\
  &\leq \mathrm{Tr}\ \frac{f\left(U^*\tilde{A}U\right)+f\left(V^*\tilde{A}V\right)}{2}-\mathrm{Tr}\ f\left(\left|\frac{U^*\tilde{A}U -V^*\tilde{A}V}{2}\right|\right)\qquad(\mbox{by Theorem \ref{lemma3}})\\
  &=\mathrm{Tr}\ \frac{U^*f(\tilde{A})U + V^*f(\tilde{A})V}{2}-\mathrm{Tr}\ f\left(\left|\frac{U^*\tilde{A}U -V^*\tilde{A}V}{2}\right|\right)\\
  &=\mathrm{Tr}\left[C^*f(A)C+Df(A)D +Cf(B)C^*\right]-
  \mathrm{Tr}\left[f\left( \left|DAC\right|\right)+ f\left(\left|C^*AD\right|\right)\right],
 \end{align*}
where the last equality follows from \eqref{qn1} and \eqref{qn2}. Putting  $B=0$ and noting that $f(0)\leq0$,  this gives the desired inequality.
\end{proof}

We remark that a non-negative superquadratic  function $f$ is convex and satisfies $f(0)=0$. If $C^*C=I$, then with $D=\sqrt{I-CC^*}$ we have $D^*D=I-CC^*\leq I$. It follows from \eqref{HPT} that
\begin{align*}
 \mathrm{Tr}f\left(C^*AC\right)+  \mathrm{Tr} f\left(D^*AD\right)\leq  \mathrm{Tr}\ C^*f(A)C +  \mathrm{Tr} D^*f(A)D.
\end{align*}
Therefore Corollary \ref{pr3} gives a refinement of \eqref{HPT}, when $f$ is a non-negative  superquadratic  function.

\bigskip

To present the second variant  of \eqref{HPT}, we give the following version of \eqref{qk} and \eqref{qkd}. The proof is similar to those of \cite[Theorem 2.1]{K} and \cite[Theorem 2.3]{KS}. We include the proof for the sake of completeness.
\begin{lemma}\label{lm5}
  Let $f:[0,\infty)\to\mathbb{R}$ be a continuous superquadratic function and $\Phi:\mathbb{M}_n\to\mathbb{M}_m$ be a unital positive linear map. If $\tau$ is an state on $\mathbb{M}_m$, then
  \begin{align*}
    f(\tau(\Phi(A)))\leq \tau (\Phi(f(A)))-\tau(\Phi(f(|A-\tau(\Phi(A))|)))
  \end{align*}
  for every positive matrix $A$.
\end{lemma}
\begin{proof}
  If $A$ is a positive matrix, then applying the functional calculus to \eqref{eq1.3} with $t=A$ and then applying the positive linear functional $\tau$  gives the inequality
\begin{align*}
\tau(f\left( A \right)) \ge f\left( s \right) + C_s \left( {\tau(A) - s}
\right) + \tau(f\left( {\left| {A - s} \right|} \right))
\end{align*}
for every $s\geq0$.    Put $s=\tau(A)$ to obtain
\begin{align}\label{qkn}
\tau(f\left( A \right)) \ge f\left( \tau(A)\right) +  \tau(f\left( {\left| {A - \tau(A)} \right|} \right)).
\end{align}
Now assume that   $\Phi:\mathbb{M}_n\to\mathbb{M}_m$ is a unital positive liner map. If $\tau$ is an state on $\mathbb{M}_m$, then the mapping $\psi_\tau:\mathbb{M}_n\to\mathbb{C}$  defined by $\psi_\tau(X)=\tau(\Phi(X))$ is an state on $\mathbb{M}_n$. Applying \eqref{qkn} to $\psi_\tau$ gives the desired inequality.
\end{proof}

The canonical trace is a  positive linear functional on $\mathbb{M}_n$. If $\tau(A)=1/n \tr (A)$, then Lemma \ref{lm5} concludes the following result.
\begin{proposition}
 Let $f:[0,\infty)\to\mathbb{R}$ be a continuous superquadratic function. If $\Phi:\mathbb{M}_n\to\mathbb{M}_m$ is a unital positive linear map, then
 \begin{align*}
   f\left(\frac{1}{n}\tr \Phi(A)\right)\leq \frac{1}{n}\tr \left[\Phi(f(A))-\Phi\left(f\left(\left|A-\frac{1}{n}\tr \Phi(A)\right|\right)\right)\right]
 \end{align*}
 for every positive matrix $A\in\mathbb{M}_n^+$.
\end{proposition}

In the next result, we present another variant of  the Hansen-Pedersen trace inequality \eqref{HPT} for superquadratic functions. We need a well-known fact from matrix analysis.

\begin{lemma}\label{lm7}\cite{bh}
If $A\in\mathbb{H}_n$ is a Hermitian matrix, then
\begin{align}\label{eq2.6}
\sum_{j=1}^{k}\lambda_j(A)=\max \sum_{j=1}^{k}\langle A\mathrm{\mathbf{u}}_j,\mathrm{\mathbf{u}}_j\rangle,\quad (k=1,\cdots,n)
\end{align}
 where the maximum is taken over all choices of orthonormal  set of vectors $\{\mathrm{\mathbf{u}}_1,\cdots,\mathrm{\mathbf{u}}_k\}$.
\end{lemma}

\begin{proposition}\label{thk11}
Let $f:[0,\infty)\to\mathbb{R}$ be a continuous superquadratic function. If $\Phi:\mathbb{M}_n\to\mathbb{M}_m$  is a unital positive linear map, then
\begin{align*}
 \mathrm{Tr}f\left(\Phi(A)\right)\leq \mathrm{Tr}\ \Phi(f(A))
  -\min\left\{\sum_{j=1}^{n}\left\langle \Phi\left(f\left(\left|A-\langle \Phi(A)\mathrm{\mathbf{u}}_j,\mathrm{\mathbf{u}}_j\rangle\right|\right)\right) \mathrm{\mathbf{u}}_j,\mathrm{\mathbf{u}}_j\right\rangle\right\},
\end{align*}
for every   positive matrix $A\in\mathbb{M}^+_n$,   where the minimum is taken over all choices of orthonormal  system  of vectors $\{\mathrm{\mathbf{u}}_1,\cdots,\mathrm{\mathbf{u}}_k\}$.
\end{proposition}
\begin{proof}
  Assume that $\lambda_1,\ldots,\lambda_n$ are eigenvalues of $\Phi(A)$ and $\{\mathrm{\mathbf{u}}_1,\cdots, \mathrm{\mathbf{u}}_n\}$ is  orthonormal  system  of corresponding eigenvectors of $\Phi(A)$. Then
  \begin{align*}
   \mathrm{Tr}f\left(\Phi(A)\right)&=\sum_{j=1}^{n}f\left(\lambda_j(\Phi(A))\right)\\
   &=\sum_{j=1}^{n}f\left(\left\langle \Phi(A) \mathrm{\mathbf{u}}_j,\mathrm{\mathbf{u}}_j\right\rangle\right)\\
   &\leq \sum_{j=1}^{n}\left[\left\langle \Phi(f(A)) \mathrm{\mathbf{u}}_j,\mathrm{\mathbf{u}}_j\right\rangle-\left\langle \Phi\left(f\left(\left|A-\langle \Phi(A)\mathrm{\mathbf{u}}_j,\mathrm{\mathbf{u}}_j\right\rangle\right|\right)  \mathrm{\mathbf{u}}_j,\mathrm{\mathbf{u}}_j\right\rangle\right]\qquad(\mbox{by \eqref{qkd}})\\
   &\leq \mathrm{Tr}\ \Phi(f(A))  -\sum_{j=1}^{n}\left\langle \Phi\left(f\left(\left|A-\langle \Phi(A)\mathrm{\mathbf{u}}_j,\mathrm{\mathbf{u}}_j\rangle\right|\right)\right) \mathrm{\mathbf{u}}_j,\mathrm{\mathbf{u}}_j\right\rangle,
  \end{align*}
  in which the last inequality follows from Lemma \ref{lm7}. This completes the proof.
\end{proof}


\section{ Klein  inequality}

In this section, we present a Klein  trace inequality for superquadratic functions. In particular, we show that if $f$ is non-negative, a refinement of the Klein  inequality \eqref{eq4.2} holds. The next lemma can be found in \cite{bh}.
\begin{lemma}\label{lm3}\cite{bh}
 If $X,Y\in\mathbb{M}_n$ are Hermitian matrices, then  the inequality
	\begin{align}\label{eq2.12}
	\mathrm{Tr} XY\leq  \langle \lambda^\downarrow(X),\lambda^\downarrow(Y)\rangle
	\end{align}
	holds.
\end{lemma}

 The main result of this section is  the following    Klein  inequality for superquadratic functions.

\begin{theorem}[Klein's Inequality for superquadratic functions]\label{thm1}
	Assume that $f:[0,\infty)\to\mathbb{R}$ is a differentiable superquadratic function with $f(0)=f'(0)=0$. Then
{\small\begin{align}\label{eq2.110}
  \mathrm{Tr}[f(A)-f(B)-(A-B)f'(B)]\geq \min\left\{\sum_{j=1}^{n}f(|x-y|); x\in\sigma(A),\ y\in\sigma(B)\right\}
\end{align}}
 for all $A,B\in\mathbb{M}_n^+$ in which $\sigma(A)$ is the set of eigenvalues of $A$.
In particular, if $f$ is non-negative,  then
	\begin{align}\label{eq2.10}
	\mathrm{Tr}[f(A)-f(B)-(A-B)f'(B)]\geq \mathrm{Tr}f(|A-B|)
	\end{align}
for all $A,B\in\mathbb{M}_n^+$.
\end{theorem}
\begin{proof}
	First we prove \eqref{eq2.10}. Suppose that $\lambda_j$ and $\mu_j$ \ $(j=1,\cdots,n)$ are eigenvalues of $A$ and $B$, respectively, arranged in decreasing order. If  $f$ is non-negative,   then $f'$ is a monotone increasing function  by Lemma \ref{lm44} and so $f'(\mu_j)$\ $(j=1,\cdots,n)$ are eigenvalues of $f(B)$ arranged in decreasing order.  Hence
	\begin{align*}
	\mathrm{Tr}(A-B)f'(B)&=\mathrm{Tr}\ Af'(B)-\mathrm{Tr}\ Bf'(B)\\
	&=\mathrm{Tr}\ Af'(B)-\sum_{j=1}^{n}\mu_jf'(\mu_j)\\
	&\leq \sum_{j=1}^{n}\lambda_jf'(\mu_j)-\sum_{j=1}^{n}\mu_jf'(\mu_j)\qquad\mbox{by \eqref{eq2.12}}\\
	&=\sum_{j=1}^{n}(\lambda_j-\mu_j)f'(\mu_j).
	\end{align*}
Moreover, it follows from proof of Theorem \ref{lemma3} that
\begin{align}
\mathrm{Tr} f(|A-B|)\leq \sum_{j=1}^{n}f\left(\left|\lambda_j-\mu_j\right|\right).
\end{align}
Note that if a   superquadratic function $f$ is  differentiable  on $(0,\infty)$ and $f(0)=f'(0)=0$, then   Lemma \ref{lemma1} implies that
	\begin{align*}
	f(t)\geq f(s)+f'(s)(t-s)+f(|t-s|)
	\end{align*}
	for all $s,t\geq0$. This gives
	\begin{align*}
	f(\lambda_j)\geq f(\mu_j)+f'(\mu_j)(\lambda_j-\mu_j)+f(|\lambda_j-\mu_j|)\qquad (j=1,\cdots,n)
	\end{align*}
	and so
	\begin{align}\label{eq2.13}
	\sum_{j=1}^{n} f(\lambda_j)\geq  \sum_{j=1}^{n} f(\mu_j)+ \sum_{j=1}^{n}f'(\mu_j)(\lambda_j-\mu_j)+ \sum_{j=1}^{n}f(|\lambda_j-\mu_j|),
	\end{align}
which proves \eqref{eq2.10}.  In general case, when $f$ is not assumed to be non-negative, we suppose that  $\lambda_j$  \ $(j=1,\cdots,n)$ are eigenvalues of $A$   arranged in decreasing order and $\mu_j$ \ $(j=1,\cdots,n)$ are eigenvalues  of $B$, arranged in such a way that $f'(\mu_1)\geq\cdots \geq f'(\mu_n)$.  By a same argument as in the first part of the proof, this guarantees the inequality $\mathrm{Tr}(A-B)f'(B)\leq \sum_{j=1}^{n}(\lambda_j-\mu_j)f'(\mu_j)$. It follows from \eqref{eq2.13} that
\begin{align*}
  \mathrm{Tr} f(A)\geq \mathrm{Tr} f(B) + \mathrm{Tr}(A-B)f'(B)+\sum_{j=1}^{n}f(|\lambda_j-\mu_j|),
\end{align*}
from which we get \eqref{eq2.110}.
 \end{proof}

When the superquadratic function $f$ is non-negative,  then Theorem~\ref{thm1} gives a refinement of the Klein's inequality \eqref{eq4.2} for convex functions. Indeed, if $f\geq0$, then
{\small \begin{align*}
	0\leq \mathrm{Tr}[f(A)-f(B)-(A-B)f'(B)-f(|A-B|)]\leq\mathrm{Tr}[f(A)-f(B)-(A-B)f'(B)].
	\end{align*}}

\begin{example}\label{ex1}
 The function $f(t)=t^p$ is superquadratic  for every  $p\geq2$.     Theorem~\ref{thm1} gives
  \begin{align*}
    0\leq \mathrm{Tr}[A^p-B^p-p(A-B)B^{p-1}-|A-B|^p]\leq\mathrm{Tr}[A^p-B^p-p(A-B)B^{p-1}]
  \end{align*}
  for  all $A,B\in\mathbb{M}_n^+$ and every $p\geq2$.

  As a simple example, assume that $p=3$ and  consider the positive matrices
  \begin{align*}
   A=\left[\begin{array}{cc}
    2 & 1\\ 1&2
    \end{array}\right]\quad\mbox{and}\quad   B=\left[\begin{array}{cc}
    2 & 0\\ 0&0
    \end{array}\right].
  \end{align*}
  Then
  \begin{align*}
\mathrm{Tr}[A^p-B^p-p(A-B)B^{p-1}]=20 \quad\mbox{and}\quad \mathrm{Tr} |A-B|^p\simeq 14.15.
  \end{align*}
\end{example}

On the other hand, if $f\geq0$ is a convex function and   $-f$ is a   superquadratic function,  then Theorem \ref{thm1} provides an upper bound for the Klein's Inequality. Applying Theorem \ref{thm1} to the superquadratic function $-f$ we obtain
\begin{align}\label{re22}
\mathrm{Tr}[f(A)-f(B)-(A-B)f'(B)]\leq \max\left\{\sum_{j=1}^{n}f(|x-y|); x\in\sigma(A),\ y\in\sigma(B)\right\}
\end{align}
 for  all $A,B\in\mathbb{M}_n^+$, while the left side is positive due to the Klein's Inequality for the convex function $f$.

\begin{example}
If $1\leq p\leq2$, then  the function $f(t)=t^p$ is convex and $-f(t)=-t^p$ is  superquadratic. It follows from \eqref{re22} that
\begin{align*}
\mathrm{Tr}[A^p-B^p-p(A-B)B^{p-1}]\leq \max\left\{\sum_{j=1}^{n}|x-y|^p; x\in\sigma(A),\ y\in\sigma(B)\right\},
\end{align*}
 for  all $A,B\in\mathbb{M}_n^+$ and every $1\leq p\leq2$.

 To see a simple example, let $p=3/2$  and consider the two matrices in Example \ref{ex1}. Then
 \begin{align*}
\mathrm{Tr}[A^p-B^p-p(A-B)B^{p-1}]\simeq3.36 \quad\mbox{and}\quad \max\left\{\sum_{j=1}^{n}|x-y|^p; x\in\sigma(A),\ y\in\sigma(B)\right\}\simeq6.19.
 \end{align*}
\end{example}

\bigskip

  If $f$ is a continuous convex function, then
$f(\langle A\mathrm{\mathbf{u}},\mathrm{\mathbf{u}}\rangle)\leq\langle f(A)\mathrm{\mathbf{u}},\mathrm{\mathbf{u}}\rangle$
for every unit vector $  \mathrm{\mathbf{u}}\in\mathbb{C}^n$, see \cite{FMPS}.  If $\{\mathrm{\mathbf{u}}_1,\cdots,\mathrm{\mathbf{u}}_n\}$ is an  orthonormal  basis   of $\mathbb{C}^n$, then
\begin{align*}
\sum_{j=1}^{n}f(\langle A\mathrm{\mathbf{u}}_j,\mathrm{\mathbf{u}}_j\rangle)&\leq \sum_{j=1}^{n} \langle f(A)\mathrm{\mathbf{u}}_j,\mathrm{\mathbf{u}}_j\rangle)\\
&\leq \sum_{j=1}^{n} \lambda_j(f(A))\qquad\mbox{by \eqref{eq2.6}}\\
&=\mathrm{Tr} f(A).
\end{align*}
In other words,
\begin{align}\label{eq2.7}
\sum_{j=1}^{n}f(\langle A\mathrm{\mathbf{u}}_j,\mathrm{\mathbf{u}}_j\rangle) \leq \mathrm{Tr} f(A).
\end{align}
for every orthonormal  basis $\{\mathrm{\mathbf{u}}_1,\cdots,\mathrm{\mathbf{u}}_n\}$ of $\mathbb{C}^n$. Inequality \eqref{eq2.7} is known as the \textit{Peierls inequality}. The equality holds in \eqref{eq2.7} when $\mathrm{\mathbf{u}}_i$'s are eigenvectors of $A$.

We present a variant of the Peierls inequality in the case when $f$ is a superquadratic function. It gives in particular a refinement of the Peierls inequality if $f$ is non-negative.
\begin{proposition}
\label{prp1}	Assume that $f$ is a superquadratic function. If $A\in\mathbb{M}^+_n$, then
	\begin{align}\label{eq2.8}
	\sum_{j=1}^{n}f(\langle A\mathrm{\mathbf{u}}_j,\mathrm{\mathbf{u}}_j\rangle)+\sum_{j=1}^{n}\langle f(|A-\langle A\mathrm{\mathbf{u}}_j,\mathrm{\mathbf{u}}_j\rangle|)\mathrm{\mathbf{u}}_j,\mathrm{\mathbf{u}}_j\rangle  \leq \mathrm{Tr} f(A)
	\end{align}
	for every orthonormal  basis $\{\mathrm{\mathbf{u}}_1,\cdots,\mathrm{\mathbf{u}}_n\}$ of $\mathbb{C}^n$. Equality holds if $f$ is non-negative and  $\mathrm{\mathbf{u}}_i$'s are eigenvectors of $A$.
\end{proposition}
\begin{proof}
	Let $f$ be a superquadratic function.  We apply the Jensen's operator inequality  \eqref{qk}  and then we use \eqref{eq2.6}. This gives \eqref{eq2.8}.

If $f$ is non-negative, then
	\begin{align*}
	\sum_{j=1}^{n}f(\langle A\mathrm{\mathbf{u}}_j,\mathrm{\mathbf{u}}_j\rangle)\leq \sum_{j=1}^{n}f(\langle A\mathrm{\mathbf{u}}_j,\mathrm{\mathbf{u}}_j\rangle)+\sum_{j=1}^{n}\langle f(|A-\langle A\mathrm{\mathbf{u}}_j,\mathrm{\mathbf{u}}_j\rangle|)\mathrm{\mathbf{u}}_j,\mathrm{\mathbf{u}}_j\rangle  \leq \mathrm{Tr} f(A).
	\end{align*}
Hence, choosing  $\mathrm{\mathbf{u}}_i$'s to be the eigenvectors of $A$, gives the equality in  \eqref{eq2.7} and so in  \eqref{eq2.8}
\end{proof}



\end{document}